\documentclass[11pt]{amsart}

\usepackage{amssymb,amsmath,amsthm} 
\usepackage{stmaryrd} 

\usepackage{enumitem} 
\setenumerate{label=(\alph*), font=\normalfont}

\usepackage[all]{xy} 

\usepackage{multirow} 
\usepackage{tabu}
\usepackage{xcolor}
\usepackage{graphicx}
\usepackage{hyperref}

\theoremstyle{plain}
\newtheorem{theorem}{Theorem}[section]
\newtheorem{lemma}[theorem]{Lemma}
\newtheorem{proposition}[theorem]{Proposition}
\newtheorem{corollary}[theorem]{Corollary}

\theoremstyle{definition}
\newtheorem{definition}[theorem]{Definition}

\theoremstyle{remark}
\newtheorem{remark}[theorem]{Remark}

\numberwithin{equation}{section}

\newcommand{\PicIm}{\operatorname{AutP}}
\newcommand{\AmExp}{\mathsf{m}}
\newcommand{\orbitsum}[2]{\Sigma_{#1}(#2)}
\newcommand{\PicForget}[2]{\overline{\operatorname{Pic}}(#1,#2)}

\begin{document}

\title{The numerical Amitsur group}
\author{Alexander Duncan}
\author{Shreya Sharma}

\address{Department of Mathematics, University of South Carolina, 
Columbia, SC 29208}
\email{duncan@math.sc.edu}
\email{shreyas@email.sc.edu}

\date{\today}

\begin{abstract}
The Amitsur subgroup of a variety with a group action
measures the failure of the action to lift to the total spaces of its
line bundles.
We introduce the ``numerical Amitsur group,'' which is an approximation
of the ordinary Amitsur subgroup that can be computed using
only the Euler-Poincar\'e characteristic on the Picard group.
As an application, we find a uniform upper bound on the exponent of the
Amitsur subgroup that depends only on the dimension and arithmetic genus
of the variety and is independent of the group.
Finally, we compute Amitsur subgroups of toric varieties
using these ideas.
\end{abstract}

\maketitle

\section{Introduction}
\label{sec:intro}

Let $X$ be a smooth projective complex variety with an action of a
finite group $G$.
A line bundle $\mathcal{L}$ is \emph{linearizable} if there exists
a linear action of $G$ on the total space of $\mathcal{L}$ that is compatible with
the action of $G$ on $X$.
A choice of such action is a \emph{linearization} of $\mathcal{L}$
and the set of isomorphism classes of line bundles together with a choice of
linearization forms a group $\operatorname{Pic}(X,G)$.

Even if $g^\ast \mathcal{L} \cong \mathcal{L}$ for every $g \in G$,
the action may not lift to the total space.
There is
an extension of $G$ by $\mathbb{C}^\times$ called the \emph{lifting
group}, which splits if and only if $\mathcal{L}$ is linearizable.
We have an exact sequence
\[
0 \to \operatorname{Hom}(G,\mathbb{C}^\times) \to
\operatorname{Pic}(X,G) \to
\operatorname{Pic}(X)^G \xrightarrow{\partial}
H^2(G,\mathbb{C}^\times) .
\]
The image of $\partial$ is called the \emph{Amitsur subgroup}
and is denoted $\operatorname{Am}(X,G)$.
Roughly speaking, the Amitsur subgroup measures the failure of line
bundles to be $G$-linearizable.
Alternatively, the Amitsur subgroup describes the set of all possible
lifting groups of line bundles for $G$ acting on $X$.

Linearizations are central in Geometric Invariant
Theory~\cite{MumFogKir94Geometric} and the \emph{theta groups} from the
theory of abelian varieties are important examples of lifting
groups~\cite{Mumford08}.
The Amitsur subgroup is an equivariant birational invariant and can be
used to understand automorphism groups of Mori fiber
spaces~\cite[Appendices]{Blanc.Cheltsov.23}.
An arithmetic version of the Amitsur subgroup sits inside the Brauer
group of the base field and measures the failure of Galois descent for
line bundles~\cite{Liedtke17}; indeed, this notion preceded the
equivariant version we study here.

\medskip

Let $J$ be a finite group acting on $\operatorname{Pic}(X)$ by group
automorphisms.
Suppose that the Euler-Poincar\'e characteristic is
$J$-invariant as a function of $\operatorname{Pic}(X)$;
in other words $\chi(gD)=\chi(D)$ for
every line bundle $D$ and element $g \in J$.
This is not a very restrictive assumption since it holds for all
actions coming from automorphisms of $X$ and all Galois group actions
coming from the arithmetic setting.
However, it excludes totally arbitrary actions such as
$\mathcal{O}(m) \mapsto \mathcal{O}(-m)$ on $\operatorname{Pic}(\mathbb{P}^n)$
that do not naturally arise from geometry.

We introduce the \emph{numerical Amitsur group}
as the quotient
\[
\operatorname{Am}^\chi(X,J) :=
\operatorname{Pic}(X)^J/\operatorname{Pic}^\chi(X,J)
\]
of $\operatorname{Pic}(X)^J$ by the subgroup
\[
\operatorname{Pic}^\chi(X,J) :=
\left\langle
\chi(D)\sum_{E \in J\cdot D} E \ \middle| \ D \in \operatorname{Pic}(X)
\right\rangle,
\]
where $\chi(D)$ is the Euler-Poincar\'e characteristic of $D$,
and $J\cdot D$ is the orbit of $D$ in $\operatorname{Pic}(X)$.

There are several reasons for introducing this notion.
First of all, we show that the numerical Amitsur group is an
``upper bound'' for the ordinary Amitsur group: 

\begin{theorem}[cf.~Theorem~\ref{thm:can_surj}] \label{thm:1}
Let $G$ be a finite group and let $X$ be a smooth projective
complex $G$-variety.  There is a canonical surjection
$\operatorname{Am}^\chi(X,G) \to \operatorname{Am}(X,G)$. 
\end{theorem}

Secondly, the numerical Amitsur group is often easier to study.
In practice, a group of automorphisms is often more complicated
than the image of the action on the Picard group.
This means that coarse estimates can be found for many groups at once.
Moreover, if $\operatorname{Pic}(X)$ and $\chi$ are well understood,
then the numerical Amitsur group can be effectively computed:

\begin{theorem}[cf.~Theorem~\ref{thm:Am_finite}] \label{thm:2}
If $\operatorname{Pic}(X)$ is finitely generated,
then there is a finite subset $S$ such that
\[
\operatorname{Pic}^\chi(X,J) :=
\left\langle
\chi(D)\sum_{E \in J\cdot D} E \ \middle| \ D \in S
\right\rangle.
\]
\end{theorem}

To demonstrate these ideas in practice, we give several applications.
We prove a uniform bound on the exponent of the Amitsur group
using only the dimension and arithmetic genus:

\begin{theorem}[cf.~Theorem~\ref{thm:uniform}] \label{thm:3}
Let $G$ be a finite group and let $X$ be a smooth projective complex
$G$-variety of dimension $n$.
The exponent of $\operatorname{Am}(X,G)$ divides
\[
(1+(-1)^np_a) \operatorname{lcm}\{1,\ldots,n+1\}
\]
where $p_a$ is the arithmetic genus of $X$.
\end{theorem}

We also demonstrate how to compute the Amitsur group for
toric varieties (Theorem~\ref{thm:toric}).
The numerical Amitsur group turns out to be a sharp upper bound
in many cases:

\begin{theorem}[cf.~Proposition~\ref{prop:AmTvsAmChi}] \label{thm:4}
Let $X$ be a smooth projective complex Fano toric variety with a reductive automorphism group.
Let $J$ be a subgroup of the image of $\operatorname{Aut}(X)$
in $\operatorname{Aut}(\operatorname{Pic}(X))$.
There exists a finite group $G$ acting faithfully on $X$ and
an isomorphism
$\operatorname{Am}^\chi(X,J) \cong \operatorname{Am}(X,G)$.
\end{theorem}

The rest of the paper is structured as follows.
In Section~\ref{sec:prelims}, we establish basic results about
linearization and the Amitsur subgroup.
In Section~\ref{sec:polynomials}, we consider numerical polynomials and
prove several technical results that will be used in subsequent
sections.
In Section~\ref{sec:upper_bounds}, we explore how the Euler-Poincar\'e
characteristic can be used to bound orders of elements in the Amitsur
group; in particular, we prove Theorem~\ref{thm:3}.
In Section~\ref{sec:numerical}, we introduce the numerical
Amitsur group and prove Theorems~\ref{thm:1}~and~\ref{thm:2}.
In Section~\ref{sec:toric}, we determine the Amitsur subgroups
of toric varieties and prove Theorem~\ref{thm:4}.
In Section~\ref{sec:examples}, we work out several toric examples in detail;
this serves both as a demonstration of the theory and to illustrate some
of its subtleties.

\section{Preliminaries}
\label{sec:prelims}

Let $G$ be a finite group.
Let $X$ be a smooth projective complex $G$-variety.
In other words, there is an action of $G$ on $X$ by morphisms of
varieties.  We do \emph{not} assume the action is faithful.

\subsection{Linearizations}

We recall the notion of $G$-linearization of a vector bundle.
See, for example,
\cite{Dol99Invariant,Brion18,MumFogKir94Geometric}.

Given vector bundles $\pi : E \to X$, $\pi' : E' \to X'$,
and a morphism $g : X \to X'$ of varieties,
an \emph{isomorphism of vector bundles lifting $g$} is
an isomorphism $\varphi_g : E \to E'$,
which is linear on fibers,
fitting into a Cartesian square
\begin{equation} \label{eq:bundle_morphism}
\xymatrix{
E \ar[r]^{\varphi_g} \ar[d] & E' \ar[d] \\
X \ar[r]^g & X'
}
\end{equation}
Equivalently, if $\mathcal{E}$ and $\mathcal{E}'$ are the corresponding
locally free sheaves, then $\varphi_g$ 
corresponds to an isomorphism $\mathcal{E} \cong g^\ast
\mathcal{E}'$ of locally free sheaves on $X$.

We will mainly be concerned with invertible sheaves and line bundles.
There is a canonical homomorphism
\[
\alpha :
\operatorname{Aut(X)} \to
\operatorname{Aut}(\operatorname{Pic}(X))
\]
given by $g \mapsto (g^\ast)^{-1}$.
We use the notation
$\PicIm(X,G)$ to denote the image of the action of $G$ on
$\operatorname{Pic}(X)$; in other words,
\[
\PicIm(X,G) := \operatorname{im} \left( G \to
\operatorname{Aut(X)} \xrightarrow{\alpha}
\operatorname{Aut}(\operatorname{Pic}(X))
\right) .
\]
We write $\PicIm(X) := \PicIm(X,\operatorname{Aut}(X))$
for the case where $G = \operatorname{Aut}(X)$.

A locally free sheaf $\mathcal{E}$ is \emph{$G$-invariant}
if $g^\ast \mathcal{E} \cong \mathcal{E}$ for every $g \in G$.
In the case where $\mathcal{L}$ is an invertible sheaf,
$\mathcal{L}$ is $G$-invariant if and only if
$[\mathcal{L}] \in \operatorname{Pic}(X)^G = \operatorname{Pic}(X)^{\PicIm(X,G)}$.

\begin{definition}
If $\mathcal{E}$ is a $G$-invariant locally free sheaf of finite rank then
the \emph{lifting group} $G_{\mathcal{E}}$ is the group of all
isomorphisms of vector bundles \eqref{eq:bundle_morphism}
lifting $g$ over all $g \in G$.
\end{definition}

The lifting group sits in a canonical exact sequence
\begin{equation} \label{eq:lifting_sequence_vb}
1 \to \operatorname{GL}_r(\mathbb{C}) \to G_{\mathcal{E}} \to G \to 1,
\end{equation}
where $\operatorname{GL}_r(\mathbb{C})$ is the group of automorphisms of
the locally free sheaf $\mathcal{E}$ of rank $r$
lifting the identity morphism of $X$.

In the case where $\mathcal{L}$ is an invertible sheaf, the canonical
sequence is
\begin{equation} \label{eq:lifting_sequence}
1 \to \mathbb{C}^\times \to G_{\mathcal{L}} \to G \to 1,
\end{equation}
where $\mathbb{C}^\times$ is a central subgroup corresponding
to the automorphisms of the corresponding line bundle $L$
over the identity morphism of $X$.

Since the lifting group $G_{\mathcal{L}}$ acts on the line bundle $L$,
each element $g$ in $G_{\mathcal{L}}$ induces a morphism
$g_\ast=(g^\ast)^{-1} : H^i(X,L) \to H^i(X,L)$ for every $i \ge 0$.
Since $g_\ast \circ h_\ast = (g \circ h)_\ast$ for the pair $(X,L)$,
this gives each $H^i(X,L)$ an action of $G_{\mathcal{L}}$.
Indeed, we have the following:

\begin{proposition}
The lifting group $G_{\mathcal{L}}$ acts on $H^i(X,\mathcal{L})$
linearly with $\mathbb{C}^\times$ in $G_{\mathcal{L}}$ identified with the
action of the nonzero scalar matrices.
If $G$ acts faithfully on $X$ and $\mathcal{L}$ is very ample,
then the action of $G_{\mathcal{L}}$ on $H^0(X,\mathcal{L})$
is faithful.
\end{proposition}

For an invertible sheaf $\mathcal{L}$ in
$\operatorname{Pic}(X)^G$ with $n+1=\dim H^0(X,\mathcal{L}) \ge 2$,
we obtain a rational map
\[
\varphi_{\mathcal{L}} : X \dasharrow \mathbb{P}^n.
\]
There is a canonical action of $G$ on $\mathbb{P}^n$
such that the rational map $\varphi_{\mathcal{L}}$ is $G$-equivariant.
Moreover, if $\mathcal{L}$ is very ample, then
the extension $G_{\mathcal{L}}$ is the pullback of
\[
1 \to \mathbb{C}^\times \to \operatorname{GL}_{n+1}(\mathbb{C}) \to \operatorname{PGL}_{n+1}(\mathbb{C}) \to 1
\]
along the embedding $G \to \operatorname{PGL}_{n+1}(\mathbb{C})$.

\begin{definition}
A $G$-invariant locally free sheaf $\mathcal{E}$ of finite rank is
\emph{$G$-linearizable} if
the extension \eqref{eq:lifting_sequence_vb} defining $G_{\mathcal{E}}$
splits.
A \emph{$G$-linearization} of $\mathcal{E}$ is a choice of $G$-action on
$\mathcal{E}$ that splits the sequence.
The set of isomorphism classes of
invertible sheaves with a choice of linearization form a group,
which we denote $\operatorname{Pic}(X,G)$.
The image of $\operatorname{Pic}(X,G)$ in
$\operatorname{Pic}(X)$ will be denoted
$\PicForget{X}{G}$.
\end{definition}

\begin{proposition}
There is a canonical exact sequence
\begin{equation} \label{eq:Leray}
\begin{gathered}
0 \to \operatorname{Hom}(G,\mathbb{C}^\times) \to
\operatorname{Pic}(X,G) \to
\operatorname{Pic}(X)^G \xrightarrow{\partial}
H^2(G,\mathbb{C}^\times),
\end{gathered}
\end{equation}
which is contravariantly functorial in both $X$ and $G$.
\end{proposition}

\begin{proof}
This is well known.
Perhaps, the most sophisticated way to see this is to use
a Leray spectral sequence of stacks
(see, e.g., \cite[3.1]{Kresch.Tschinkel22} and
\cite[1.3]{Pirutka.Zhang24}).
We will simply recall concrete interpretations of the maps
leaving the remaining details to the reader.

Recall that the set of splittings of \eqref{eq:lifting_sequence}
is given by the group cohomology group $H^1(G,\mathbb{C}^\times)$.
Since the extension is central, this is canonically isomorphic to
$\operatorname{Hom}(G,\mathbb{C}^\times)$.
The first morphism takes the set of
splittings of $G_{\mathcal{O}_X}$ to the corresponding linearization.

The next morphism is the ``forgetful'' morphism which forgets the
$G$-linearization of an invertible sheaf.

Recall that $H^2(G,\mathbb{C}^\times)$ is canonically isomorphic to
the group of extensions of $G$
by $\mathbb{C}^\times$.  Therefore, for a $G$-invariant invertible sheaf
$\mathcal{L}$,
the element $\partial([\mathcal{L}])$ is simply the class of the
extension \eqref{eq:lifting_sequence} defining the lifting group
$G_{\mathcal{L}}$.
\end{proof}

Since $H^2(G,\mathbb{C}^\times)$ is finite and torsion, it is
immediately clear that for every $G$-invariant invertible sheaf
$\mathcal{L}$, there exists some positive integer $d$ such that $d
\partial(\mathcal{L})=0$.

\begin{definition}
Suppose $\mathcal{L}$ is a $G$-invariant invertible sheaf on $X$.
The \emph{Amitsur period of $\mathcal{L}$},
denoted $\AmExp(X,G;\mathcal{L})$, is the order of the element
$\partial(\mathcal{L})$ in the group $H^2(G,\mathbb{C}^\times)$.
We use the shorthand $\AmExp(\mathcal{L})$ when the $G$-variety is
clear.
The \emph{Amitsur period of the $G$-variety $X$},
denoted $\AmExp(X,G)$, is the least common multiple
of $\AmExp(X,G;\mathcal{L})$ over all line bundles
$[\mathcal{L}] \in \operatorname{Pic}(X)^G$.
\end{definition}

An arithmetic version of the following proposition, in the case of
$H^0$, can be found in 
\cite[Proposition 7.1.15(i)]{Colliot-Thelene.Skorobogatov21}.

\begin{proposition} \label{prop:coh_subspace}
If $[\mathcal{L}]$ is in $\operatorname{Pic}(X)^G$ and there exists a
$G_{\mathcal{L}}$-subrepresentation $V$ of $H^i(X,\mathcal{L})$
of dimension $d$, then $d\partial(\mathcal{L})=0$.
In other words, $\AmExp(\mathcal{L})$ divides $d$.
\end{proposition}

\begin{proof}
Let $G_{d\mathcal{L}}$ be the group in the extension
$d\partial(\mathcal{L})=\partial(\mathcal{L}^{\otimes d})$
in $H^2(G,\mathbb{C}^\times)$.
In other words, we have the following commutative diagram of exact
sequences:
\[
\xymatrix{
1 \ar[r] &
\mathbb{C}^\times \ar[r] \ar[d]^{d} &
G_{\mathcal{L}} \ar[d] \ar[r] &
G \ar[r] \ar@{=}[d] &
1\\
1 \ar[r] &
\mathbb{C}^\times \ar[r]^{\alpha} &
G_{d\mathcal{L}} \ar[r] &
G \ar[r] &
1
}
\]
where the morphism $\mathbb{C}^\times \to \mathbb{C}^\times$
is the power map $x \mapsto x^d$.
We want to show the bottom sequence splits.

We have a $d$-dimensional representation $V$ of $G_\mathcal{L}$
whose restriction to $\mathbb{C}^\times$ corresponds to multiplication by scalar
matrices.
The representation $\Lambda^d V$ is a $1$-dimensional representation
of $G_\mathcal{L}$ whose restriction to $\mathbb{C}^\times$ is
the power map $x \mapsto x^d$.
Therefore, we have a homomorphism
$G_{\mathcal{L}} \to \mathbb{C}^\times$
which factors through $G_{d\mathcal{L}}$ in the diagram above.
The resulting quotient morphism $G_{d\mathcal{L}} \to \mathbb{C}^\times$
is a retract for the morphism
$\alpha : \mathbb{C}^\times \to G_{d\mathcal{L}}$.
The kernel of this retract gives the desired
splitting $G \to G_{d\mathcal{L}}$.
\end{proof}

The following fact, apparently first observed by A.~Kuznetsov,
is the foundation for this paper.
An arithmetic version of this can be found in
\cite[Proposition 7.1.15(ii)]{Colliot-Thelene.Skorobogatov21}.

\begin{proposition} \label{prop:Euler}
Let $\chi(\mathcal{L})$ be the Euler-Poincar\'e characteristic of
$\mathcal{L}$.  We have $\chi(\mathcal{L})\partial(\mathcal{L})=0$. 
In other words, $\AmExp(\mathcal{L})$ divides $\chi(\mathcal{L})$.
\end{proposition}

\begin{proof}
In view of Proposition~\ref{prop:coh_subspace}, we have
$\left(\dim H^i(X,\mathcal{L})\right)\partial(\mathcal{L})=0$
for every $i$.  Since $\chi(\mathcal{L})$ is
an alternating sum of $\dim H^i(X,\mathcal{L})$, the result follows
immediately.
\end{proof}

From \cite[Proposition~2.11]{Blanc.Cheltsov.23}, we have the following.

\begin{proposition}
The canonical line bundle has a canonical linearization.
\end{proposition}

\subsection{Amitsur subgroup}

Our main interest is the \emph{Amitsur subgroup}
\begin{align*}
\operatorname{Am}(X,G) := &\operatorname{im}\left( \partial
: \operatorname{Pic}(X)^G \to H^2(G,\mathbb{C}^\times)
\right) \\
\cong &\operatorname{coker}\left(
 \operatorname{Pic}(X,G) \to \operatorname{Pic}(X)^G
\right)
\end{align*}
first named in \cite{Blanc.Cheltsov.23}.
This is an equivariant analog of the arithmetic version first named in
\cite{Liedtke17}, although its study goes back decades.

In this paper, we are less interested in the specific embedding into
$H^2(G,\mathbb{C}^\times)$ and are more interested in its description as
a quotient of $\operatorname{Pic}(X)^G$.
Therefore, we will more frequently refer to it as the
\emph{Amitsur group} rather than the Amitsur subgroup.

The arithmetic interest is related to the fact that it a birational
invariant.  This is true in the equivariant context as well
by \cite[Theorem A.1]{Blanc.Cheltsov.23}:

\begin{theorem}
If $G$ acts faithfully on $X$, then
$\operatorname{Am}(X,G)$ is a $G$-equivariant birational invariant of
smooth projective $G$-varieties.
\end{theorem}

The following Theorem is due to Dolgachev~\cite{Dol99Invariant}:

\begin{theorem} \label{thm:Dolgachev}
If $X$ is a curve and $G$ acts faithfully, then
$\operatorname{Am}(X,G)=H^2(G,\mathbb{C}^\times)$.
\end{theorem}

Suppose $G$ and $H$ are finite groups
and $\varphi : G \to H$ is a group homomorphism.
Suppose $X$ is a $G$-variety, $Y$ is an $H$-variety
and $f : X \to Y$ is a $\varphi$-equivariant morphism.
In other words, $f(g\cdot x)=\varphi(g)f(x)$ for all $g \in G$.
Then there is an induced map
\[
f^\ast : \operatorname{Am}(Y,H)
\to \operatorname{Am}(X,G)
\]
using the functoriality of the exact sequence
\eqref{eq:Leray}.
Indeed, $\operatorname{Am}(-,-)$ is a contravariant
functor from the category of varieties with a finite group action to the
category of abelian groups.

\begin{proposition}
If $f : X \to Y$ is a $G$-equivariant morphism,
then the induced morphism
$f^\ast : \operatorname{Am}(Y,G) \to \operatorname{Am}(X,G)$
is injective.  If $f$ has a $G$-equivariant section, then
$\operatorname{Am}(Y,G)$ is a direct summand.
\end{proposition}

\begin{proof}
See \cite[Lemma A.6]{Blanc.Cheltsov.23},
where faithfulness is assumed in the statement, but not needed in the proof.
\end{proof}

Even restriction to a subgroup $H \subset G$ can be a subtle operation.
Indeed, in Remark~\ref{rem:MackeyFunctor} below,
we see that it is possible for $\operatorname{Am}(X,G)$ to be trivial,
while $\operatorname{Am}(X,H)$ is nontrivial.  However, we have the
following. 
Recall that $\PicIm(X,G)$ is the image of
the $G$-action on $\operatorname{Pic}(X)$.

\begin{proposition} \label{prop:PicIm}
Suppose that $\varphi : H \to G$ is a group homomorphism and
$\PicIm(X,G)=\PicIm(X,H)$.
Then the induced homomorphism $\operatorname{Am}(X,G) \to \operatorname{Am}(X,H)$
is surjective.
\end{proposition}

\begin{proof}
This follows by the description of $\operatorname{Am}(X,G)$ as the
cokernel of the map
\[
\operatorname{Pic}(X,G) \to \operatorname{Pic}(X)^G.
\]
By our hypothesis, $\operatorname{Pic}(X)^G=\operatorname{Pic}(X)^H$.
A line bundle is $H$-linearizable if it is $G$-linearizable,
so the image of $\operatorname{Pic}(X,G)$ is contained in the image
of $\operatorname{Pic}(X,H)$.
\end{proof}

\begin{proposition}
If $\varphi : H \to G$ is a group homomorphism then the kernel of
the homomorphism $\operatorname{Am}(X,G) \to \operatorname{Am}(X,H)$
is the set of classes $[\mathcal{L}]$ where $\mathcal{L}$ is a
$G$-invariant line bundle on $X$ such that $\varphi$ factors through
$G_{\mathcal{L}} \to G$.
\end{proposition}

\begin{proof}
The kernel of the morphism
$H^2(G,\mathbb{C}^\times) \to H^2(H,\mathbb{C}^\times)$
is the set of classes that are trivialized when pulled back along
$\varphi$.
\end{proof}

\section{Numerical Polynomials}
\label{sec:polynomials}

We recall a special case of Snapper's Theorem~\cite{Kleiman66}.
Given a set of line bundles
$\mathcal{L}_1,\ldots,\mathcal{L}_r$ on a variety $X$,
the Euler-Poincar\'e characteristic
\[
(a_1,\ldots,a_r) \mapsto \chi\left(\mathcal{L}_1^{\otimes a_1}
\otimes \cdots \otimes \mathcal{L}_r^{\otimes a_r}\right)
\]
is a numerical polynomial.
In this section, we recall some well known facts about numerical
polynomials as well as prove some new facts which will be important
below.

We recall that a \emph{numerical polynomial} or \emph{integer-valued
polynomial} is a polynomial $f \in \mathbb{Q}[x_1,\ldots,x_r]$
such that $f(x_1,\ldots,x_r)$ is an integer for all tuples
$(x_1,\ldots,x_r) \in \mathbb{Z}^r$.
In this section, we write $\mathbf{x}=(x_1,\ldots,x_r)$ for elements
of $\mathbb{Z}^r$.

Recall that the binomial coefficient
\[
\binom{x}{a} = \frac{x(x-1)\cdots (x-a+1)}{a!}
\]
is a numerical polynomial in $\mathbb{Q}[x]$ of degree $a$.
For $\mathbf{x} \in \mathbb{Z}^r$ and $\mathbf{a} \in \mathbb{N}^r$
we write
\[
\binom{\mathbf{x}}{\mathbf{a}}
:= \binom{x_1}{a_1} \cdots \binom{x_r}{a_r}
\]
where each $\binom{x_i}{a_i}$ is a binomial coefficient.

If $f$ is a numerical polynomial of degree $n$ in $r$ variables,
then there exists a unique expression
\[
f(\mathbf{x}) = \sum_{\mathbf{a}}
c_{\mathbf{a}} \binom{\mathbf{x}}{\mathbf{a}}
\]
where the sum runs over all $\mathbf{a}$ satisfying the relations
\[
0 \le a_1,\ 0 \le a_2,\ \ldots, 0 \le a_r,\ \sum_{i=1}^r a_i \le
n
\]
and $c_{\mathbf{a}}$ are integers depending on $f$.

For each index $i$, we have the forward difference operator
\[
(\Delta_i f) (\mathbf{x}) = f(\mathbf{x}+\mathbf{e}_i) - f(\mathbf{x})
\]
where $\mathbf{e}_i$ is the $i$th standard basis vector in
$\mathbb{Z}^r$.
These operators can be iterated to obtain
\[
(\Delta_i^a f) (\mathbf{x}) = \sum_{k=0}^a
(-1)^{a-k} \binom{a}{k} f(\mathbf{x}+k\mathbf{e}_i)
\]
For $\mathbf{a} \in \mathbb{N}^r$, we define the operator
\[
\Delta^{\mathbf{a}} f := \Delta_1^{a_1} \cdots \Delta_r^{a_r} f .
\]

We recall the multivariate Newton-Gregory interpolation formula:

\begin{proposition}
If $f$ is an integer-valued polynomial in
$r$ variables of degree $n$,
then
\[
f(\mathbf{y}+\mathbf{x}) = \sum_{\mathbf{a}}
\binom{\mathbf{x}}{\mathbf{a}}(\Delta^{\mathbf{a}}f)(\mathbf{y})
\]
where the sum is over $\mathbf{a} \in \mathbb{N}^r$ such that
$\sum_{i=1}^r a_i \le n$. 
\end{proposition}

The following is the key new concept we need for producing finite
generating sets for the numerical Amitsur group.

\begin{definition}
A subset $S \subseteq \mathbb{Z}^r$ is
\emph{integrally-poised for polynomials of degree $n$}
if there exists a set
of numerical polynomials $\{ c_{\mathbf{a}} \}_{\mathbf{a} \in S}$
such that
\[
f(\mathbf{x}) = \sum_{\mathbf{a} \in S} c_{\mathbf{a}}(\mathbf{x})
f(\mathbf{a})
\]
for all polynomials $f$ of degree $n$ and all $\mathbf{x} \in \mathbb{Z}^r$.
\end{definition}

Our main example of an integrally-poised subset is the lattice simplex:

\begin{proposition} \label{prop:simplex_poised}
The set
\[
S_n = \left\{
\mathbf{a} \in \mathbb{Z}^r\ \middle| \
a_1 \ge 0,\ \ldots,\ a_r \ge 0,\ \sum_{i=1}^r a_i \le n
\right\} 
\]
is integrally-poised for polynomials of degree $n$.
\end{proposition}

\begin{proof}
Suppose $f(\mathbf{x})$ is a multivariate integer-valued polynomial in
$r$ variables of degree $n$.
From above, we see that  $(\Delta^{\mathbf{a}} f)(0)$ is a specific integral linear
combination of $f(\mathbf{b})$ where $0 \le b_i \le a_i$
for every index $i$.
From the Newton-Gregory interpolation formula,
we see that
\[
f(\mathbf{x}) = \sum_{\mathbf{a} \in S_n}
c_\mathbf{a}(\mathbf{x})f(\mathbf{a})
\]
where each $c_\mathbf{a}$ is a multivariate
integer-valued polynomial that only depends on the number of variables
$r$ and degree $n$.
\end{proof}

Finally, we obtain the result needed for our application to Amitsur
groups.

\begin{proposition} \label{prop:poised}
Suppose $S \subseteq \mathbb{Z}^r$ is integrally poised for
polynomials of degree $n+1$.
If $g$ is an integer-valued polynomial degree $n$, then
we have equality of the subgroups
\begin{equation} \label{eq:gen_sub_equality}
\langle g(\mathbf{x})\mathbf{x} \mid  \mathbf{x} \in \mathbb{Z}^n \rangle
=
\langle g(\mathbf{x})\mathbf{x} \mid \mathbf{x} \in S \rangle
\end{equation}
of $\mathbb{Z}^r$.
\end{proposition}

\begin{proof}
By assumption, we have polynomials $c_{\mathbf{a}}(\mathbf{x})$
such that 
\[
f(\mathbf{x}) = \sum_{\mathbf{a} \in S} c_{\mathbf{a}}(\mathbf{x})
f(\mathbf{a})
\]
for every polynomial $f$ of degree $n+1$.
Applying the formula to $f(\mathbf{x})=g(\mathbf{x})x_i$ for each
$i=1,\ldots, r$ we recover the formula
\[
g(\mathbf{x})\mathbf{x} = \sum_{\mathbf{a} \in S} c_{\mathbf{a}}(\mathbf{x})
g(\mathbf{a})\mathbf{a} .
\]
This shows that every generator on the left hand side
of \eqref{eq:gen_sub_equality} is contained in the right hand side.
\end{proof}

\subsection{Univariate Polynomials}

Let $f(x)$ be a univariate numerical polynomial of degree $n$.
Recall that $f$ has a unique expression
\[
f(x) = \sum_{i=0}^n a_i \binom{x}{i}
\]
for integers $a_0,\ldots, a_n$.  By convention, we set $a_i=0$ for all
$i < 0$ and $i > n$.

\begin{lemma} \label{lem:vals_to_coefs}
$\gcd \{ f(k) \} := \gcd \{ f(k) \mid k \in \mathbb{Z}\} = \gcd\{ a_0,\ldots,a_n \}$
\end{lemma}

\begin{proof}
From Prop~\ref{prop:simplex_poised}, we see that
$\gcd \{ f(k) \mid k \in \mathbb{Z}\}$ is equal to the expression
$\gcd \{ f(0),f(1),\ldots,f(n) \}$.
Consider the equations
\begin{equation} \label{eq:binomial_transform}
f(j) = \sum_{i=0}^n a_i \binom{j}{i}
\end{equation}
for integers $0 \le j \le n$.
This a binomial transform,
thus each $a_i$ can be written as an integral
linear combination of $f(0),\ldots,f(n)$.
In particular, the two sets have the same $\gcd$.

For those not familiar with binomial transforms, we sketch an
argument.
We can view \ref{eq:binomial_transform} as a matrix equation where we've multiplied
a vector with entries $a_i$ by a matrix with entries $\binom{j}{i}$
to obtain a vector with entries $f(j)$.
Since $\binom{j}{i} = 0$ for all $j < i$ and
$\binom{i}{i}=1$ for all $i$, the matrix is invertible over the integers.
\end{proof}

\begin{lemma} \label{lem:kfk_expression}
\[kf(k) = \left(\sum_{i=0}^n i(a_i+a_{i-1}) \binom{k}{i}\right) +
(n+1)a_n \binom{k}{n+1}\]
\end{lemma}

\begin{proof}
Rearranging the identity
\[
\binom{k}{i+1} = \frac{k-i}{i+1}\binom{k}{i} 
\]
we obtain
\[
k\binom{k}{i} = (i+1)\binom{k}{i+1} + i\binom{k}{i} .
\]
The result follows immediately.
\end{proof}

Applying Lemma~\ref{lem:vals_to_coefs} to
Lemma~\ref{lem:kfk_expression}, we obtain:

\begin{corollary} \label{cor:mod_condition}
An integer $m$ divides $\gcd \{ kf(k) \mid k \in \mathbb{Z}\}$ if and only if the integers
$a_0,\ldots, a_n$ satisfy the congruences
\begin{align*}
a_1 + a_0 &\equiv 0 \mod{m}\\
2(a_2 + a_1) &\equiv 0 \mod{m}\\
3(a_3 + a_2) &\equiv 0 \mod{m}\\
&\vdots\\
k(a_k+a_{k-1}) &\equiv 0 \mod{m}\\
&\vdots\\
n(a_n+a_{n-1}) &\equiv 0 \mod{m}\\
(n+1)a_n &\equiv 0 \mod{m}.
\end{align*}
\end{corollary}

\begin{lemma} \label{lem:Cheb_bound}
$\displaystyle
\frac{\gcd \{ kf(k) \}}{\gcd \{ f(k) \}}$
divides 
$\operatorname{lcm}\{1,\ldots,n+1\}$.
\end{lemma}

\begin{proof}
It follows immediately from the definitions that
$\gcd \{ f(k) \}$ divides $\gcd \{ kf(k) \}$.
By Lemma~\ref{lem:vals_to_coefs}, we see that
\[
\gcd\left\{
\frac{a_0}{\gcd \{ f(k) \}},
\cdots
\frac{a_n}{\gcd \{ f(k) \}}
\right\} = 1;
\]
thus, we may assume $a_0,\ldots,a_n$ have no common factor.
Let $p$ be a prime and $s$ be a positive integer
such that $p^s$ divides $\gcd \{ kf(k) \}$.
It remains to prove that $p^s \le n+1$.

Suppose otherwise; that $p^s > n+1$.
Consider $0 < k \le n+1$.
Write $k=p^rm$ where $r$ is a non-negative integer and $m$
is a positive integer coprime to $p$.
Since $k < p^s$, we have $r < s$.
Applying Corollary~\ref{cor:mod_condition}, we have the following:
\begin{align*}
ka_k &\equiv -ka_{k-1} \pmod{p^s}\\
p^rma_k &\equiv -p^rma_{k-1} \pmod{p^s}\\
ma_k &\equiv -ma_{k-1} \pmod{p^{s-r}}\\
ma_k &\equiv -ma_{k-1} \pmod{p}\\
a_k &\equiv -a_{k-1} \pmod{p}.
\end{align*}
From $k=n+1$, we conclude that $a_n$ is divisible by $p$
since $a_{n+1}=0$.
From $k=n$, we then conclude that $a_{n-1}$ is also divisible by $p$.
Continuing in this way, we conclude that $a_0,\ldots,a_n$ are all divisible by $p$.
This contradicts that they have no common factor.
\end{proof}

\begin{remark}
The expression
\[
\operatorname{lcm}\{1,\ldots,n+1\}
\]
is exactly the exponent of the symmetric group $S_{n+1}$.
This function has many other interpretations~\cite[A003418]{oeis}.
In particular, the logarithm $\psi(x) = \log \operatorname{lcm}\{1,\ldots,x\}$
is the second Chebyshev function.
One can show that $\psi(x) \sim x$ and certain more refined estimates 
turn out to be equivalent to the Riemann
Hypothesis~\cite[\S{18}]{Davenport}.
(We thank F.~Thorne for pointing out this reference.)
\end{remark}

\section{Upper bounds on Amitsur periods}
\label{sec:upper_bounds}

Suppose $X$ is a smooth projective variety with an action of a finite
group $G$.
Recall that, if $\mathcal{L}$ is $G$-invariant line bundle,
then the Amitsur period $\AmExp(X,G;\mathcal{L})=\AmExp(\mathcal{L})$
is the order of $\partial[\mathcal{L}]$ in
the Amitsur group $\operatorname{Am}(X,G)$.
The value $\AmExp(X,G)$ is the exponent of the entire group
$\operatorname{Am}(X,G)$.

\begin{lemma} \label{lem:mdivf}
If $X$ is a smooth $G$-variety and $\mathcal{L}$ is a $G$-invariant
line bundle class,
then $\AmExp(\mathcal{L})$ divides $\gcd \{ kf(k) \}$
where $f(k)= \chi(\mathcal{L}^{\otimes{k}})$.
\end{lemma}

\begin{proof}
Let us use additive notation with a divisor $D$
representing the class of $\mathcal{L}$;
thus $f(k)=\chi(kD)$ for every integer $k \in \mathbb{Z}$.
By Proposition~\ref{prop:Euler},
we have $\chi(\mathcal{L})\partial(\mathcal{L})=0$ for every
$G$-invariant line bundle $\mathcal{L}$.
Thus
\[
\partial(f(k)kD)=\chi(kD)\partial(kD)=0
\]
for every integer $k$.
In other words, $f(k)k[D]$ is $G$-linearizable for every $k \in
\mathbb{Z}$.
This means that
\[
A = \left\langle
kf(k)[D]
\ \middle|\ k \in \mathbb{Z}
\right\rangle
\]
is a subgroup of the image $\PicForget{X}{G}$ of $\operatorname{Pic}(X,G)$ in
$\operatorname{Pic}(X)^G$.
Now $A$ is a subgroup of $\langle [D] \rangle \cong \mathbb{Z}$.
Thus $A$ is generated by $\gcd \{ kf(k) \}[D]$.
Since
\[
\PicForget{X}{G} \cap \langle [D] \rangle = \left\langle \AmExp(D)[D] \right\rangle,
\]
we conclude that $\AmExp(D)$ divides $\gcd \{ kf(k) \}$.
\end{proof}

\subsection{A uniform bound}

We are now in a position to prove the uniform bound stated in the
introduction.

\begin{theorem} \label{thm:uniform}
Let $X$ be a smooth variety of dimension $n$
with arithmetic genus $p_a$.
For any finite group $G$, the Amitsur period
$\AmExp(X,G)$ divides
\[
(1 + (-1)^n p_a) \operatorname{lcm}\{1,\ldots,n+1\} .
\]
\end{theorem}

\begin{proof}
The choice of finite group is irrelevant.
We consider the polynomial $f(k)=\chi(kD)$ where $D$ represents a
$G$-invariant divisor class.
Observe that $f(0)=\chi(\mathcal{O}_X)=(1 + (-1)^n p_a)$.
Therefore $\gcd \{ f(k) \}$ divides $(1 + (-1)^n p_a)$.
From Lemma~\ref{lem:Cheb_bound},
$\gcd\{kf(k)\}$ divides $\gcd\{f(k)\} \operatorname{lcm}\{1, \dots , n+1\}$,
which, in turn, divides $(1+(-1)^np_a) \operatorname{lcm}\{1, \dots , n+1\}$.
Now from Lemma~\ref{lem:mdivf}, $\AmExp(X,G)$ divides $\gcd\{kf(k)\}$.
Thus, $\AmExp(X,G)$ divides
$(1+(-1)^np_a) \operatorname{lcm}\{1, \dots , n+1\}$.
\end{proof}

\begin{remark}
For any positive integer $m$,
there exists a choice of finite group $G$ such that
$\operatorname{Am}(\mathbb{P}^{m-1},G) \cong \mathbb{Z}/m\mathbb{Z}$.
If $n \ge m-1$, then we may construct
$X \cong \mathbb{P}^{m-1} \times \mathbb{P}^{n-m+1}$
where $G$ acts trivially on $\mathbb{P}^{n-m+1}$.
We see that $X$ has dimension $n$ and
$\operatorname{Am}(X,G) \cong \mathbb{Z}/m\mathbb{Z}$ has exponent $m$.
This construction works for any $m \le n+1$;
in particular for any prime power $\le n+1$.
Therefore, for varieties of genus $0$,
the bound from Theorem~\ref{thm:uniform} is multiplicatively sharp.
\end{remark}

\begin{remark}
The bound is vacuous for curves of genus $1$
since $\chi(\mathcal{O}_X)=0$.
This is not surprising since there is no uniform bound for \emph{any}
such curve if one does not fix the group $G$.
Indeed, for every elliptic curve and every prime $p$,
there is a group of automorphisms
$G$ isomorphic to $C_p^2$ that acts by translations.
Recall that $H^2(G,\mathbb{C}^\times) \cong
\mathbb{Z}/p\mathbb{Z}$.
By Theorem~\ref{thm:Dolgachev},
we see that $\operatorname{Am}(X,G) = \mathbb{Z}/p\mathbb{Z}$.
Therefore, the period $\AmExp(X,G)$ can be made arbitrarily large
by appropriate choice of $G$.
\end{remark}

\begin{remark}
The bound is \emph{not} necessarily sharp for higher genus curves.
By Theorem~\ref{thm:Dolgachev}, if $G$ acts faithfully then
$\AmExp(X,G)$ is the exponent of $H^2(G,\mathbb{C}^\times)$.
The Hurwitz bound states that $|G| \le 84(p_a - 1)$ for a group acting
faithfully on a curve $X$.
However, if the $p$-primary component of $H^2(G,\mathbb{C}^\times)$ is non-zero,
then $G$ must have order at least $p^2$.
\end{remark}

\subsection{Bounds on the Amitsur period for specific divisors}

The remainder of this section demonstrate how one can use the numerical
results from the previous section to bound the Amitsur periods of
specific divisors using intersection-theoretic information.
In order to do this, we use expressions
\begin{equation} \label{eq:f_coefs}
f(k) = \chi(kD) = \sum_{i=0}^n a_i \binom{k}{i}
\end{equation}
where $a_0,\ldots,a_n$ are constants depending on $D$.

\begin{proposition}
Suppose $X$ is a smooth curve with genus $p_a$
and $D$ is a divisor whose class is in $\operatorname{Pic}(X)^G$.
\begin{enumerate}
\item $\AmExp(D)$ divides $\gcd \{ kf(k) \}$.
\item $\gcd \{ f(k) \} = \gcd\{(1-p_a), \deg(D) \}$
\item $\gcd \{ kf(k) \} = \gcd\{\deg(D)+(1-p_a), 2\deg(D)\}$
\item
$\displaystyle
\frac{\gcd \{ kf(k) \}}{\gcd \{ f(k) \}} = \begin{cases}
2 & \textrm{$1-p_a$ and $\deg(D)$ have same parity}\\
1 & \textrm{otherwise} .
\end{cases}$.
\end{enumerate}
\end{proposition}

\begin{proof}
The Riemann-Roch theorem gives us
\[
\chi(kD)=\deg(kD)+(1-p_a)
=\left(\deg(D)\right)k+(1-p_a) .
\]
Thus, $a_0=(1-p_a)$ and $a_1=\deg(D)$ in \eqref{eq:f_coefs}
and the results follow from the results in the previous section.
\end{proof}

\begin{proposition}
Suppose $X$ is a smooth surface with arithmetic genus $0$
and $D$ is a divisor whose class is in $\operatorname{Pic}(X)^G$.
\begin{enumerate}
\item $\AmExp(D)$ divides $6$.
\item $\AmExp(D)$ divides $\chi(D)$.
\item If $2$ divides $\AmExp(D)$, then
$D^2$ and $D\cdot K_X$ are both even, but are distinct modulo $4$.
\item If $3$ divides $\AmExp(D)$, then
$D\cdot K_X \equiv 0 \mod{3}$ and $D^2 \equiv 1 \mod{3}$.
\end{enumerate}
\end{proposition}

\begin{proof}
Since $\chi(0)=1+p_a=1$, we conclude $\AmExp(D)$ divides $6$ from
Theorem~\ref{thm:uniform}.
The second bound is Proposition~\ref{prop:Euler},
but we include it here for completeness.
For the remaining bounds, we use the fact that $\AmExp(D)$
divides $\gcd \{ kf(k) \}$ by Lemma~\ref{lem:mdivf}.

The usual Riemann-Roch theorem for surfaces is
\[
\chi(D) =  1+p_a + \frac{1}{2}\left(D^2-D\cdot K_X\right)
\]
From this, we determine the coefficients
\begin{align*}
a_0 &= \chi(0) = 1+p_a = 1\\
a_1 &= \frac{1}{2}\left(D^2-D\cdot K_X\right)\\
a_2 &= D^2
\end{align*}
for the binomial coefficients in \eqref{eq:f_coefs}.
Thus, by Corollary~\ref{cor:mod_condition}, the multiplicity
$\AmExp(D)$ divides all of the following expressions:
\begin{align}
a_1+a_0 &= \frac{1}{2}\left(D^2-D\cdot K_X\right)+1 \label{eq:2a10}\\
2(a_2+a_1) &= 3D^2-D\cdot K_X \label{eq:2a21}\\
3a_2 &= 3D^2 \label{eq:2a32}
\end{align}

Suppose $\gcd \{ kf(k)\}$ is divisible by $2$.
Immediately, we see that $D^2$ must be even.
Using expression \eqref{eq:2a21}, we conclude that $D \cdot K_X$ is also
even.
Next, expression \eqref{eq:2a10} implies that $\frac{1}{2}\left(D^2-D\cdot K_X\right)$
is odd.

Suppose $\gcd \{ kf(k)\}$ is divisible by $3$.
From expression \eqref{eq:2a21}, we see that $D\cdot K_X \equiv 0 \mod{3}$.
Now, from expression \eqref{eq:2a10}, we have $D^2+2 \equiv 0 \mod{3}$.
\end{proof}

\begin{proposition}
Suppose $X$ is a smooth threefold with arithmetic genus $0$
and $D$ is a divisor whose class is in $\operatorname{Pic}(X)^G$.
\begin{enumerate}
\item $\AmExp(D)$ divides $12$.
\item $\AmExp(D)$ divides $\chi(D)$.
\item $2D^3 \equiv 6 \pmod{m}$ and $D^2\cdot K_X \equiv 4 \pmod{2m}$
\end{enumerate}
\end{proposition}

\begin{proof}
As before, the first two bounds follow from
Theorem~\ref{thm:uniform} and Proposition~\ref{prop:Euler}.
Using Hirzebruch-Riemann-Roch for threefolds, we find
\[
\chi(kD) = \frac{1}{6}(kD)^3 -\frac{1}{4}(kD)^2\cdot K_X
+\frac{1}{12}(kD)\cdot\left(K_X^2+c_2\right)+1.
\]
Viewing this as a polynomial in $k$ in the binomial basis,
we obtain the following coefficients
\begin{align*}
a_0 &= 1\\
a_1 &= \frac{1}{6}D^3 -\frac{1}{4}D^2\cdot K_X
+\frac{1}{12}D\cdot\left(K_X^2+c_2\right) \\
a_2 &= D^3-\frac{1}{2}D^2\cdot K_X\\
a_3 &= D^3
\end{align*}
as in \eqref{eq:f_coefs}.
Thus, $\AmExp(D)$ divides all of the following expressions:
\begin{align}
a_1+a_0&=
\frac{1}{6}D^3 -\frac{1}{4}D^2\cdot K_X
+\frac{1}{12}D\cdot\left(K_X^2+c_2\right)+1 \label{eq:3a10}\\
2(a_2+a_1)&=
\frac{7}{3}D^3 -\frac{3}{2}D^2\cdot K_X
+\frac{1}{6}D\cdot\left(K_X^2+c_2\right) \label{eq:3a21}  \\
3(a_3+a_2)&=
6D^3-\frac{3}{2}D^2\cdot K_X \label{eq:3a32}\\
4a_3&=
4D^3. \label{eq:3a43}
\end{align}
Let $m$ be a divisor of $\AmExp(D)$.
Subtracting twice \eqref{eq:3a10} from \eqref{eq:3a21},
and multiplying \eqref{eq:3a32} by $2$, we obtain the congruences
\begin{align}
\chi(D) &\equiv 0 \mod{m} \label{eq:3c1}\\
2D^3 -D^2\cdot K_X-2 &\equiv 0 \mod{m} \label{eq:3c2}\\
12D^3-3D^2\cdot K_X &\equiv 0 \mod{2m} \label{eq:3c3}\\
4D^3 &\equiv 0 \mod{m}  \label{eq:3c4}.
\end{align}
Taking twice the sum of \eqref{eq:3c2} and \eqref{eq:3c4}, then subtracting
\eqref{eq:3c3}, we find
\[
D^2\cdot K_X \equiv 4 \mod{2m}.
\]
Substituting this back into the congruences above, we find that
\[
4D^3 \equiv 0 \pmod{m} \textrm{ and }
2D^3 \equiv 6 \pmod{m}.
\]
Since $m$ divides $12$, the first congruence is redundant.
\end{proof}

\section{Numerical Amitsur group}
\label{sec:numerical}

Suppose $X$ is a smooth projective variety.
Let $G$ be a finite group acting on $\operatorname{Pic}(X)$
by group automorphisms such that $\chi(D)=\chi(g^\ast
D)$ for all $g \in G$ and $[D] \in
\operatorname{Pic}(X)$.
The reader should imagine that $G$ is $\PicIm(X,G)$
from above, although
we do not necessarily assume that $G$ comes from an action on $X$.

We use the following shorthand for orbit sums.
For $[D] \in \operatorname{Pic}(X)$, write
\[
\orbitsum{G}{D} := \sum_{[E] \in G\cdot [D]} [E]
\]
where $G\cdot [D]$ is the $G$-orbit of $[D]$ in $\operatorname{Pic}(X)$.
Alternatively,
\[
\orbitsum{G}{D} = \sum_{g \in G/H} g[D]
\]
where $H$ is the stabilizer of $[D]$ in $G$.
Observe that $\orbitsum{G}{D}$ is always in $\operatorname{Pic}(X)^G$
even though $[D]$ may not be.

\begin{definition}
The \emph{numerical Amitsur group}
$\operatorname{Am}^\chi(X,G)$ is the cokernel in the short exact sequence
\begin{equation} \label{eq:num_Am_def}
0 \to \operatorname{Pic}^\chi(X,G) \to \operatorname{Pic}(X)^G
\to \operatorname{Am}^\chi(X,G) \to 0 .
\end{equation}
where
$\operatorname{Pic}^\chi(X,G)$ is the subgroup
\begin{equation} \label{eq:numPic_def}
\operatorname{Pic}^\chi(X,G) :=
\left\langle \chi(D)\orbitsum{G}{D} \ \middle| \
[D] \in \operatorname{Pic}(X) \right\rangle .
\end{equation}
\end{definition}

A divisor class $[D]$ is \emph{numerically $G$-split}
if and only if $[D] \in \operatorname{Pic}^\chi(X,G)$.
(We anticipate that $\operatorname{Pic}^\chi(X,G)$ will also be useful in
the arithmetic setting, which explains the terminology ``split'' instead
of ``linearizable.'')

The notation $\operatorname{Pic}^\chi(X,G)$ suggests that it is a
``numerical'' version of $\operatorname{Pic}(X,G)$, but the latter is
not a subgroup of $\operatorname{Pic}(X)$.
Instead, $\operatorname{Pic}^\chi(X,G)$ is more analogous to the group
$\PicForget{X}{G}$
of isomorphism classes of linearizable line bundles
where the particular choice of linearization is not part of the data.

We now show, as stated in the introduction, that
the numerical Amitsur group is an ``upper bound'' for the
ordinary Amitsur group.

\begin{theorem} \label{thm:can_surj}
There is a canonical surjection
\[
\operatorname{Am}^\chi(X,\PicIm(X,G)) \to \operatorname{Am}(X,G)
\]
for every finite group $G$ acting on $X$.
\end{theorem}

\begin{proof}
Using \eqref{eq:Leray},
it suffices to show that every element of
$\operatorname{Pic}^\chi(X,G)$ is linearizable.
Suppose $[D]$ is in $\operatorname{Pic}^\chi(X,G)$.
Since tensor products of linearizable invertible sheaves are
linearizable, it suffices to assume that
\[
[D] = \chi(E)
\sum_{[F] \in G\cdot[E]} [F]
\]
for some divisor class $[E] \in \operatorname{Pic}(X)$.
Let $H$ be the stabilizer of $[E]$ in $\operatorname{Pic}(X)$.
Then $\chi(E)[E]$ is $H$-linearizable by
Proposition~\ref{prop:coh_subspace}.
Thus $[D]$ is $G$-linearizable by Lemma~\ref{lem:orbit_sum} below.
\end{proof}

\begin{lemma} \label{lem:orbit_sum}
Suppose $H$ is the stabilizer in $G$ of $[D] \in \operatorname{Pic(X)}$.
If $D$ is $H$-linearizable, then $\orbitsum{G}{D}$ is
$G$-linearizable.
\end{lemma}

\begin{proof}
Let $\mathcal{L}$ be the corresponding invertible sheaf.
Let $t_1,\ldots,t_s$ be a system of distinct representatives for
the left cosets in $G/H$
and define
\[
\mathcal{F} := \bigoplus_{i=1}^s t_i^\ast \mathcal{L} .
\]
If $\mathcal{F}$ has a linearization, then so must its determinant.
Observe that the determinant is $\orbitsum{G}{D}$.
Thus, it remains to exhibit a $G$-linearization on $\mathcal{F}$.
We do so using a construction inspired from the ``induced
representation'' from ordinary representation theory.

We change to the language of total spaces. 
Let $L$ (resp. $F$) be the total space of $\mathcal{L}$
(resp. $\mathcal{F}$) and let $\pi : L \to X$ be the projection.
Since $\mathcal{L}$ is $H$-linearizable,
there exists an action on $H$ on $L$ 
such that $\pi$ is $H$-equivariant.
Let $L^s$ be the $s$-fold product of $L$ (as a variety)
and let $p_i : L^s \to L$ denote the $i$th projection.
For each $g \in G$ and index $i$ there is a unique index $g(i)$ and
element $h_{g,i} \in H$
such that $gt_i = t_{g(i)}h_{g,i}$.
We define an action of $G$ on $L^s$ by requiring that
\[
p_i(g\cdot \ell) :=  h_{g,g^{-1}(i)}p_{g^{-1}(i)}(\ell)
\]
for each index $i$.
One checks that this gives a well-defined $G$-action
since the $H$-action is well-defined.

Observe that $F$ is the fibred product $L \times_X \cdots \times_X L$
for the morphisms $t_1 \circ \pi, \ldots, t_s \circ \pi : L \to X$.
Equivalently, $F$ is the subvariety of $L^s$ consisting of
elements $\ell \in L^s$ such that
\[
t_1(\pi(p_1(\ell))) = \cdots = t_s(\pi(p_s(\ell)))
\]
as elements of $X$.
Applying the $G$-action to $\ell \in F$, we see that
\begin{align*}
t_i(\pi(p_i(g\cdot\ell)))
&= t_ih_{g,g^{-1}(i)}(\pi(p_{g^{-1}(i)}(\ell)))\\
&= gt_{g^{-1}(i)}(\pi(p_{g^{-1}(i)}(\ell)))\\
&= gt_{i}(\pi(p_{i}(\ell)))
\end{align*}
for every index $i$.
This means that $g \cdot \ell \in L^s$ also represents an element
from $F$ for every $g$.
It also means that $G$ has an action on $F$ making the projection
$F \to X$ equivariant.
In other words, $\mathcal{F}$ has a $G$-linearization.
\end{proof}

The following theorem shows that, at least when $\operatorname{Pic}(X)$
is finitely generated, the numerical Amitsur group can actually be
computed.

\begin{theorem} \label{thm:Am_finite}
Suppose $\operatorname{Pic}(X)$ is finitely generated.
Suppose $G$ is a finite group acting on $\operatorname{Pic}(X)$
such that $\chi$ is $G$-invariant.
There exists a finite set $S \subset \operatorname{Pic}(X)$ such that
\[
\operatorname{Pic}^\chi(X,G) =
\left\langle \chi(D)\orbitsum{G}{D} \ \middle| \
[D] \in S \right\rangle .
\]
\end{theorem}

\begin{proof}
Let $H$ be a fixed subgroup of $G$ and consider
\[
Q_H :=
\left\langle \chi(D) \orbitsum{G}{D} \ \middle|\
D \in \operatorname{Pic}(X)^H
\right\rangle .
\]
The group $\operatorname{Pic}(X)^H$ is a finitely generated free abelian
group.
Thus, for each subgroup $H$, the restriction of $\chi$
to $\operatorname{Pic}(X)^H$ can be viewed as a polynomial of fixed degree
in a finite number of variables.
Thus, by Proposition~\ref{prop:simplex_poised} there exists a
finite set $S_H$ such that
\[
Q_H = \left\langle \chi(D) \orbitsum{G}{D} \ \middle|\
D \in S_H
\right\rangle
\]

Now observe that
\[
\operatorname{Pic}^\chi(X,G) = \sum_{H \le G} M_H
\]
where
\[
M_H :=
\left\langle \chi(D) \orbitsum{G}{D} \ \middle|\
D \in \operatorname{Pic}(X),\
\operatorname{Stab}_G([D]) = H
\right\rangle .
\]
for each subgroup $H \le G$.
Note that $M_H \subseteq Q_H$, but we do not have equality in
general.

If $[D] \in \operatorname{Pic}(X)$ and
$H$ is a subgroup of
$K=\operatorname{Stab}_G([D])$,
then
\[
\sum_{g \in G/H} g[D] = [K:H]\sum_{g \in G/K} g[D]
=[K:H] \orbitsum{G}{D}.
\]
From this, we know that $Q_H \subseteq \sum_{H \le K} M_K$
for all subgroups $H$ of $G$.
Thus,
\[
\sum_{H \le G} M_H
 = \sum_{H \le G} Q_H
\]
and we conclude that
\[
\operatorname{Pic}^\chi(X,G) = \sum_{H \le G}
\left\langle \chi(D) \orbitsum{G}{D} \ \middle|\
D \in S_H
\right\rangle.
\]
The set $S = \bigcup_{H \le G} S_H$ is the desired set of
generators.
\end{proof}

\begin{remark}
It is interesting to consider which properties
of $\operatorname{Am}(X,G)$ hold for
$\operatorname{Am}^\chi(X,G)$ as well.
In particular, the numerical Amitsur group is \emph{not} a birational invariant
even in the case where $G$ is trivial (consider
$\mathbb{P}^2$ and $\mathbb{P}^1 \times \mathbb{P}^1$).
Additionally, it is not clear whether the canonical bundle $K_X$
is numerically split for every $X$, even though it is always
linearizable.
\end{remark}

\section{Toric varieties}
\label{sec:toric}

Let $X$ be a smooth projective toric variety~\cite{CoxLitSch11Toric}. 
Suppose $T$ is the torus, $\Sigma$ is the fan, and
$M:=\operatorname{Hom}(T,\mathbb{C}^\times)$ is the character
lattice.
We have an exact sequence
\begin{equation} \label{eq:toric_Pic}
0 \to M \to \operatorname{TDiv}(X) \to \operatorname{Pic}(X) \to 0.
\end{equation}
where $\operatorname{TDiv}(X) \cong \mathbb{Z}^{\Sigma(1)}
\cong \mathbb{Z}^r$ is a free
abelian group with basis indexed by the rays
$\rho_1,\ldots,\rho_r$ of $\Sigma$.
Each ray corresponds to an irreducible $T$-invariant divisor
of $X$.
Let $D_1,\ldots,D_s \in \operatorname{TDiv}(X)$ be a system of distinct
representatives for the divisor classes in $\operatorname{Pic}(X)$
corresponding to these rays.
The relation between $r$ and $s$ is 
\[
n_1 + \dots + n_s = r
\]
where $n_i$ is the number of rays in each divisor class $[D_i]$.

From \cite{Cox95The-homogeneous}, we recall the Cox ring and its
connection to the automorphism group of $X$.
The Cox ring $\operatorname{Cox}(X)$ of $X$ is a
$\operatorname{Pic}(X)$-graded polynomial ring $\mathbb{C}[x_1,\ldots,x_r]$
where each $x_i$ has the grading of the
corresponding class in $\operatorname{Pic}(X)$.
For every divisor class $[D] \in \operatorname{Pic}(X)$, we have
\[
H^0(X,\mathcal{O}_X(D)) \cong \operatorname{Cox(X)}_{[D]}
\]
where $\operatorname{Cox(X)}_{[D]}$ denotes the homogeneous component
of the Cox ring of degree $[D]$.

There is an exact sequence
\[
1 \to S \to \widetilde{\operatorname{Aut}}(X) \to \operatorname{Aut}(X)
\to 1
\]
where $S =
\operatorname{Hom}(\operatorname{Pic}(X),\mathbb{C}^\times)$
is the torus dual to the Picard group
and $\widetilde{\operatorname{Aut}}(X)$ acts by polynomial automorphisms of
$\operatorname{Cox}(X)$ that are compatible with the grading.

\begin{lemma} \label{lem:toric_lifting}
Suppose $G$ is a finite group acting on $X$ and let
$\widetilde{G}$ be the pullback to $\widetilde{\operatorname{Aut}}(X)$.
Let $\mathcal{L}$ be an effective line bundle in $\operatorname{Pic}(X)^G$
and let $\operatorname{ev}_{\mathcal{L}} : S \to \mathbb{C}^\times$
be the corresponding homomorphism.
Then the extension defining the lifting group $G_{\mathcal{L}}$ is the image of
the extension defining $\widetilde{G}$ along the induced morphism
$(\operatorname{ev}_{\mathcal{L}})_\ast:
H^2(G,S) \to H^2(G,\mathbb{C}^\times)$.
\end{lemma}

\begin{proof}
This amounts to proving that there is a commutative diagram with exact rows
\[
\xymatrix{
1 \ar[r] &
S \ar[r] \ar[d]^{\operatorname{ev}_{\mathcal{L}}} &
\widetilde{G} \ar[r] \ar[d] &
G \ar[r] \ar@{=}[d] &
1 \\
1 \ar[r] &
\mathbb{C}^\times \ar[r] &
G_{\mathcal{L}} \ar[r] &
G \ar[r] &
1
}
\]
Since $\mathcal{L}$ is $G$-invariant, the action of $\widetilde{G}$
leaves stable the homogeneous component $\operatorname{Cox}(X)_{[\mathcal{L}]}$.
Since $\mathcal{L}$ is effective, $\operatorname{Cox}(X)_{[\mathcal{L}]}$ is
non-zero
and $\lambda \in S$ acts on $v \in \operatorname{Cox}(X)_{[\mathcal{L}]}$ via
$\lambda \cdot v =
\operatorname{ev}_{\mathcal{L}}(\lambda)v$.
Since $\operatorname{Cox}(X)_{[\mathcal{L}]} \cong H^0(X,\mathcal{L})$,
the result follows.
\end{proof}

From \cite[Theorem 4.5]{Duncan16}, the morphism
$\alpha : \operatorname{Aut(X)} \to \PicIm(X)$
has a splitting.  Thus, we have the description
\[
\widetilde{\operatorname{Aut}}(X) \cong
U \rtimes \left(\prod_{i=1}^s \operatorname{GL}_{n_i}(\mathbb{C}) \right) \rtimes
\PicIm(X)
\]
where $U$ is a unipotent group.
Each
$\operatorname{GL}_{n_i}(\mathbb{C})$ acts on the linear span $V_i$ of the
$n_i$ variables $\{x_j\}$ with degree $[D_i]$.
The group $\PicIm(X)$ acts by permuting the subspaces $\{V_i\}$.
We have the description
\[
\operatorname{Aut}(\Sigma) = \left(\prod_{i=1}^s S_{n_i}\right) \rtimes
\PicIm(X)
\]
where each $S_{n_i}$ permutes the variables within each $V_i$.
Taking the dual of \eqref{eq:toric_Pic}
we obtain an exact sequence 
\[
1 \to S \to \widetilde{T} \to T \to 1
\]
of tori where $\widetilde{T} \cong \left(\mathbb{C}^\times\right)^r$ acts
by scalar multiplication on each $x_i$.

\begin{definition}
Suppose $J$ is a finite subgroup of $\PicIm(X)$.  Define
\[
\operatorname{Pic}^T(X,J) := \langle n_i \orbitsum{J}{[D_i]} \mid\ 0 \le
i \le s \rangle \subseteq \operatorname{Pic}(X)^J,
\]
where we recall that $n_i$ is the number of rays in the class $[D_i]$,
and define
\[
\operatorname{Am}^T(X,J) :=
\frac{\operatorname{Pic}(X)^J}{\operatorname{Pic}^T(X,J)}.
\]
\end{definition}

\begin{lemma} \label{lem:toric_coh_char}
Let $J$ be a subgroup of $\PicIm(X)$ and let
$W$ be the preimage of $J$ in $\operatorname{Aut}(\Sigma)$.
There exist isomorphisms
\begin{equation} \label{eq:cohom_interp}
H^1(W,M) \cong \operatorname{coker} \left(
\operatorname{TDiv}(X)^W
\to \operatorname{Pic}(X)^J \right) \cong
\operatorname{Am}^T(X,J).
\end{equation}
\end{lemma}

\begin{proof}
Observe that $\operatorname{TDiv}(X)$ is a permutation $W$-lattice;
in other words, $W$ acts by permutations of a basis.
By Shapiro's Lemma, this means $H^1(W,\operatorname{TDiv}(X))=0$.
We apply group cohomology $H^i(W,-)$ to \eqref{eq:toric_Pic} and obtain
\[
0 \to M^W \to \operatorname{TDiv}(X)^W \to \operatorname{Pic}(X)^J
\to H^1(W,M) \to 0 .
\]
This establishes the first isomorphism.

Let $H=\prod_{i=1}^s S_{n_i}$ denote the kernel of $W \to J$.
The group $H$ acts on the set of rays $\{\rho_i\}$ by all permutations
that leave invariant the divisor classes.
Therefore, the $H$-orbit of $D_i$ is equal to the set of basis
elements of $\operatorname{TDiv}(X)$ with degree $[D_i]$.
Consequently,
\[
\beta_1 := \orbitsum{H}{D_1},\ \ldots,\ \beta_s := \orbitsum{H}{D_s}
\]
is a basis for $\operatorname{TDiv}(X)^H$.
Therefore we have
\[
\operatorname{TDiv}(X)^W = \left(\operatorname{TDiv}(X)^H\right)^J
\]
where $J$ acts by permutations on $\beta_1,\ldots,\beta_s$.

The divisor classes are all distinct, so the stabilizer of $J$
acting on $\beta_i$ is the same as that of $[D_i]$ in $\operatorname{Pic}(X)$.
Therefore
\[
[\orbitsum{W}{D_i}] = [\orbitsum{J}{\beta_i}]
=\orbitsum{J}{[\beta_i]} = n_i\orbitsum{J}{[D_i]} .
\]
We conclude that the image of $\operatorname{TDiv}(X)^W$
is equal to $\operatorname{Pic}^T(X,J)$.
The second isomorphism now follows.
\end{proof}

\begin{theorem} \label{thm:toric}
For any finite subgroup $G$ of $\operatorname{Aut}(X)$
such that $\PicIm(X,G)=J$, there exists a canonical surjection
$\operatorname{Am}^T(X,J) \to \operatorname{Am}(X,G)$.
\end{theorem}

\begin{proof}
Both groups are quotients of $\operatorname{Pic}(X)^J$
so it suffices to show that
\[
\operatorname{Pic}^T(X,J) \subseteq \PicForget{X}{G} .
\]
Thus, we only need to show that
each $n_i\orbitsum{J}{[D_i]}$ is $G$-linearizable.

Let $\widetilde{G}$ be the preimage of $G$ in
$\widetilde{\operatorname{Aut}}(X)$.
Note that $\widetilde{G}$ is an extension of $G$ by the torus $S$.
Since $G$ is finite, $\widetilde{G}$ is reductive.
Thus, we may conjugate by an element of $U$ to ensure
$\widetilde{G}$ is contained in
\[
R := \left(\prod_{i=1}^s \operatorname{GL}_{n_i}(\mathbb{C}) \right)
\rtimes \PicIm(X).
\]
(Alternatively, conjugation by an element of $U$ amounts to a different choice
of torus $T$.)

Observe that $R$ acts on $\operatorname{Pic(X)}$ through
$\operatorname{Aut}(X)$.
Let $\widetilde{H}$ be the stabilizer in $\widetilde{G}$ of $[D_i]$
and let $H:=\widetilde{H}/S$ be its image in $\operatorname{Aut}(X)$.
Since $R$ acts on the Cox ring, we see that $\widetilde{H}$ factors
through the lifting group $H_{[D_i]}$ by Lemma~\ref{lem:toric_lifting}.
Observe that $\widetilde{H}$ leaves invariant the subspace $V_i$ of
$\operatorname{Cox}(X)_{[D_i]}$ spanned by the variables $x_j$ with
degree $[D_i]$.
Therefore, there is a $n_i$-dimensional subrepresentation of the lifting
group $H_{[D_i]}$ in $H^0(X,\mathcal{O}_X(D_i))$.
Therefore, $n_i[D_i]$ is $H$-linearizable by
Proposition~\ref{prop:coh_subspace}.

By Lemma~\ref{lem:orbit_sum},
we conclude that $n_i\orbitsum{G}{[D_i]}$ is $G$-linearizable.
Since $J$ is simply the image of $G$, we have
$n_i\orbitsum{G}{[D_i]}=n_i\orbitsum{J}{[D_i]}$
and the result is proven.
\end{proof}

\begin{proposition} \label{prop:AmT_optimal}
For every subgroup $J$ of $\PicIm(X)$,
there exists a finite subgroup $G$ of $\operatorname{Aut}(X)$
such that $\PicIm(X,G)=J$ and
$\operatorname{Am}^T(X,J) \cong \operatorname{Am}(X,G)$.
\end{proposition}

\begin{proof}
Let $W$ be the preimage of $J$ in $\operatorname{Aut}(\Sigma)$
and let $w=|W|$.  Let $H$ be the $w$-torsion subgroup of $T$;
thus, $H \cong (\mathbb{Z}/w\mathbb{Z})^n$.
The group we will use is $G := H \rtimes W$.

It suffices to show that
\[
\PicForget{X}{G} \subseteq \operatorname{Pic}^T(X,J) 
\]
as subgroups of $\operatorname{Pic}(X)^J$.
Let $\mathcal{L}$ be a $G$-linearizable line bundle.
We will prove that $\mathcal{L}$ is linearly equivalent to a divisor
from $\operatorname{TDiv}(X)^W$, which will establish that
$[\mathcal{L}]$ is in $\operatorname{Pic}^T(X,J)$ via
Lemma~\ref{lem:toric_coh_char}.

Let $G_{\mathcal{L}}$ be the lifting group for $G$.
Taking the preimage of $H$, we obtain
the lifting group $H_{\mathcal{L}}$ as a subgroup of $G_{\mathcal{L}}$.
We see that $G_{\mathcal{L}}/H_{\mathcal{L}} \cong G/H \cong W$.
Let $\widetilde{G}$ be the pullback of $G$ to
$\widetilde{\operatorname{Aut}}(X)$.
Since $\operatorname{Aut}(\Sigma)$ acts on the Cox ring,
the surjection $\widetilde{G} \to W$ has a splitting.
Therefore, by Lemma~\ref{lem:toric_lifting}, the surjection
$G_{\mathcal{L}} \to W$ has a splitting.
Thus $G_{\mathcal{L}} \cong H_{\mathcal{L}} \rtimes W$.

Let $T_{\mathcal{L}}$ be the lifting group of the torus $T$.
Technically, we have only defined lifting groups for finite groups,
but we only need it here to satisfy
a version of Lemma~\ref{lem:toric_lifting}.
Specifically, we can just set $T_{\mathcal{L}}$
to be the quotient of $\widetilde{T}$ by the kernel of
$\operatorname{ev}_{\mathcal{L}} : S \to \mathbb{C}^\times$.
We have
$(T \rtimes W)_{\mathcal{L}} \cong T_{\mathcal{L}} \rtimes W$
by similar reasoning as the previous paragraph.

We obtain a commutative diagram with exact rows:
\begin{equation} \label{eq:big_one}
\xymatrix{
1 \ar[r] &
S \ar[r] \ar@{->>}[d]^{\operatorname{ev}_{\mathcal{L}}} &
\widetilde{T} \rtimes W \ar[r] \ar@{->>}[d] &
T \rtimes W \ar[r] \ar@{=}[d] &
1 \\
1 \ar[r] &
\mathbb{C}^\times \ar[r] &
T_{\mathcal{L}} \rtimes W \ar[r] &
T \rtimes W \ar[r] &
1\\
1 \ar[r] &
\mathbb{C}^\times \ar[r] \ar@{=}[u] &
H_{\mathcal{L}} \rtimes W \ar[r] \ar@{^{(}->}[u] &
H \rtimes W \ar[r] \ar@{^{(}->}[u] &
1
}
\end{equation}
In the above, the morphisms $H \rtimes W \to T \rtimes W$
and $H_{\mathcal{L}} \rtimes W \to T_{\mathcal{L}} \rtimes W$
are injective.
The bottom row is the defining sequence for the lifting group
$G_{\mathcal{L}} \cong H_{\mathcal{L}} \rtimes W$.
Since $\mathcal{L}$ is $G$-linearizable, the
bottom row is a split exact sequence.

Let $P :=
\operatorname{Hom}(T_{\mathcal{L}},\mathbb{C}^\times)$
and $Q := \operatorname{Hom}(H_{\mathcal{L}},\mathbb{C}^\times)$
be the character groups, with the induced $W$-action.
Consider \eqref{eq:big_one} where we only take the normal
subgroup $N$ from each entry $N \rtimes W$, but remember that
all the morphisms are $W$-equivariant.
Next, we apply the character group functor
$\operatorname{Hom}(-,\mathbb{C}^\times)$
to the resulting diagram.
We obtain a commutative diagram of $W$-modules with exact rows:
\begin{equation} \label{eq:big_two}
\xymatrix{
0 \ar[r] &
M \ar[r] &
\operatorname{TDiv}(X) \ar[r] &
\operatorname{Pic}(X) \ar[r] &
0 \\
0 \ar[r] &
M \ar[r] \ar@{->>}[d]^q \ar@{=}[u] &
P \ar[r] \ar@{->>}[d] \ar@{^{(}->}[u] &
\mathbb{Z} \ar[r] \ar@{=}[d] \ar@{^{(}->}[u]^{[\mathcal{L}]} &
0 \\
0 \ar[r] &
M/wM \ar[r] &
Q \ar[r] &
\mathbb{Z} \ar[r] &
0
}
\end{equation}
Exactness on the right follows in the first two rows since
$\operatorname{Hom}(-,\mathbb{C}^\times)$ is
exact when restricted to tori.
The bottom row is exact on the right since the morphism $P \to Q$ is
surjective.
Moreover, the bottom row is a split exact sequence of $W$-modules
since the bottom row of \eqref{eq:big_one} is a split exact sequence.

Recall that the group $\operatorname{Ext}^1_{\mathbb{Z}W}(\mathbb{Z},M)$
of extensions of $W$-modules
is naturally isomorphic to $H^1(W,M)$.
Let $\eta \in H^1(W,M)$ denote the extension defining $P$
in \eqref{eq:big_two}
and let $\nu \in H^1(W,M/wM)$ denote the extension defining $Q$.
We have an exact sequence
\[
0 \to M \xrightarrow{-\times w} M
\xrightarrow{q} M/wM \to 0
\]
of $W$-modules where $\nu = q_\ast(\eta)$.
Recall from \cite[Corollary III.10.2]{Bro82Cohomology}
that multiplication by $w=|W|$ induces the zero
map on group cohomology $H^1(W,-)$.
Thus the induced long exact sequence
\[
\cdots \to H^1(W,M) \xrightarrow{-\times w} H^1(W,M)
\xrightarrow{q_\ast} H^1(W,M/wM) \to \cdots
\]
establishes that $q_\ast$ is injective.
Since the last row of \eqref{eq:big_two} splits,
we see that $\nu=0$ and, therefore $\eta=0$.
Thus the middle row of \eqref{eq:big_two} also splits.

Looking at the
top right square of \eqref{eq:big_two}, we use the splitting to produce
a $W$-equivariant composition
\[
\mathbb{Z} \to P \to \operatorname{TDiv}(X) \to \operatorname{Pic}(X)
\]
such that $1 \mapsto [\mathcal{L}]$.
Since $\mathbb{Z}$ has the trivial action, we conclude that
$\mathcal{L}$ is in the image of $\operatorname{TDiv}(X)^W$
as desired. 
\end{proof}

The groups $\operatorname{Am}^T(X,J)$
and $\operatorname{Am}^\chi(X,J)$
are both approximations for the usual Amitsur group
that exploit Proposition~\ref{prop:coh_subspace}
in similar ways.
However,
they are not the same in general (see Section \ref{sec:Hirzebruch} below).
The basic reason for this is that the invariants
$\dim H^0(X,D_i)$, $\chi(D_i)$ and $n_i$ are not equal in general.
However, we have the following.

\begin{proposition} \label{prop:AmTvsAmChi}
Let $X$ be a smooth Fano toric variety
with reductive automorphism group.
Then $\operatorname{Am}^T(X,J) \cong \operatorname{Am}^\chi(X,J)$
for every finite subgroup $J$ of $\PicIm(X)$.
\end{proposition}

\begin{proof}
Recall that the canonical bundle is
\[
K_X = -\sum_{i=1}^r D_{\rho_i}
\]
where
$D_{\rho_1}, \ldots, D_{\rho_r}$ are the divisors corresponding to
the rays of $X$.
Since $-K_X$ is ample, using a vanishing theorem of
M.~Musta\c{t}\u{a}~\cite[Theorem~9.3.7]{CoxLitSch11Toric} we see that
\[
H^p\left(X,\mathcal{O}_X\left(D_{\rho_j}\right)\right)=
H^p\left(X,\mathcal{O}_X\left(-K_X -\sum_{i\neq j} D_{\rho_i}\right)\right)
=0
\]
for every index $1 \le j \le r$ and every $p > 0$.
Therefore,
\[
\chi(D_{\rho_j})=\dim H^0(X,\mathcal{O}_X(D_{\rho_j}))
\] for every divisor $D_{\rho_j}$ coming from a ray.

We recall the notion of Demazure roots in
\cite[\S{4}]{Cox95The-homogeneous}.
The unipotent radical of $\operatorname{Aut}(X)$ is nontrivial
if and only if there exists a Demazure root that is \emph{not} semisimple.
Since $\operatorname{Aut}(X)$ is reductive, this means
all the Demazure roots are semisimple.
By \cite[Lemma~4.4]{Cox95The-homogeneous}, we conclude that
the only monomials in $\operatorname{Cox}(X)$ with degree equal to
that of the generators $x_1,\ldots,x_r$ are
the generators themselves.  Therefore,
\[
H^0(X,\mathcal{O}_X(D_i)) \cong \operatorname{Cox}(X)_{[D_i]}
\]
and we conclude that $\chi(D_i)=n_i$ for all $1 \le i \le s$.

Therefore, the defining generators of $\operatorname{Pic}^T(X,J)$
are contained in $\operatorname{Pic}^\chi(X,J)$
and we have a chain of surjections
\[
\operatorname{Am}^T(X,J) \to \operatorname{Am}^\chi(X,J)
\to \operatorname{Am}(X,G)
\]
for every finite group $G$ with $\PicIm(X,G)=J$.
By Proposition~\ref{prop:AmT_optimal}, there exists a choice of $G$
such that they are in fact isomorphisms.
\end{proof}

Various conditions for a smooth Fano toric variety to have a
reductive automorphism group can be found in \cite{Nill06}. 

\section{Examples}
\label{sec:examples}

\subsection{Products of projective spaces}
\label{sec:projective_products}

Let $n$ and $m$ be positive integers and consider $X=(\mathbb{P}^n)^m$.
Viewing $X$ as a toric variety, one finds
\[
M \cong \mathbb{Z}^{nm},\
\operatorname{TDiv}(X) \cong \mathbb{Z}^{(n+1)m},\
\textrm{and }
\operatorname{Pic}(X) \cong \mathbb{Z}^m.
\]
The basis of torus-invariant divisors are partitioned by linear
equivalence into $m$ sets of size $n+1$.  The full automorphism group is
\[
\operatorname{Aut}(X) \cong \operatorname{PGL}_{n+1}(\mathbb{C}) \rtimes S_m
\]
while the automorphisms of the fan are given by
\[
\operatorname{Aut}(\Sigma) = (S_{n+1})^m \rtimes
S_m.
\]
Thus, $\PicIm(X) = S_m$.  The automorphism groups are always reductive.

Suppose $J \subseteq S_m$.  Then $W=(S_{n+1})^m \rtimes J$ and we find
\[
\operatorname{TDiv}(X)^W = (\mathbb{Z}^{(n+1)m})^W
= \left(\left((\mathbb{Z}^{n+1})^{S_{n+1}}\right)^m\right)^J
= \left(\mathbb{Z}^m\right)^J .
\]
The map into $\operatorname{Pic}(X)^J$ is simply multiplication by $n+1$.
Thus, by Theorem~\ref{thm:toric}, we find:
\[
\operatorname{Am}^T(X,J) \cong
\operatorname{coker}
\left( \operatorname{TDiv}(X)^W \to \operatorname{Pic}(X)^J\right)
\cong \left((\mathbb{Z}/(n+1)\mathbb{Z})^m\right)^J .
\]

When $m=1$ we recover the fact that,
for any finite group $G$ acting on $\mathbb{P}^n$, the group
$\operatorname{Am}(\mathbb{P}^n,G)$ is a cyclic group of order dividing
$n+1$.

\subsection{Hirzebruch surfaces}
\label{sec:Hirzebruch}

Let $X$ be the Hirzebruch surface of degree $e$.
In other words, $X$ is a ruled surface over $\mathbb{P}^1$ with a
section of self-intersection $-e$.
Let $H$ be the class of the section and $F$ be the class of a fiber.
In this case, $X$ is a toric surface with maximal rays
\[
D_1=\begin{pmatrix} 1\\0 \end{pmatrix},\
D_2=\begin{pmatrix} 0\\1 \end{pmatrix},\
D_3=\begin{pmatrix} -1\\e \end{pmatrix},\
D_4=\begin{pmatrix} 0\\-1 \end{pmatrix}
\]
where $[D_1]=F$, $[D_2]=H$, $[D_3]=F$, and $[D_4]=H+eF$.
In this case, $\operatorname{Aut}(\Sigma) \cong S_2$ interchanges $D_1$ and
$D_3$.  Thus $\PicIm(X)$ is trivial in this case.
We compute that
\[
\operatorname{TDiv}(X)^W = \langle D_1+D_3, D_2, D_4 \rangle,
\]
which has image $\langle 2F, H, H+eF \rangle$ in $\operatorname{Pic}(X)$.
We conclude that, for $J=\PicIm(X)=1$, we have
\[
\operatorname{Am}^T(X,J)\cong
\begin{cases}
0 & \textrm{if $e$ is odd,}\\
\mathbb{Z}/2\mathbb{Z} & \textrm{if $e$ is even.}
\end{cases}
\]

Hirzebruch surfaces are not necessarily Fano and
their automorphism groups are typically not reductive.
Thus, Theorem~\ref{prop:AmTvsAmChi} does not apply
and the numerical Amitsur group may have a different
structure.  Indeed, we see this in the following.

\begin{proposition}
Let $J =\PicIm(X)=1$.
If $e$ is odd, then $\operatorname{Am}^\chi(X,J)$ is trivial.
If $e$ is even, then $\operatorname{Am}^\chi(X,J)\cong
(\mathbb{Z}/2)^2$.
\end{proposition}

\begin{proof}

The class of the section $H$ and fiber $F$ form a basis for the Picard
group with intersection theory $F^2 = 0$, $H^2=-e$, and $F\cdot H=1$.
We have $K_X = -2H-(2+e)F$, so if $D=aH+bF$, then
\[
\chi(D) = -e\binom{a}{2}+ab+(1-e)a+b+1 .
\]

We compute $\chi(D)D$ for $D \in \{H,F,H+F,H-F\}$ and determine that
\[
\begin{pmatrix} 2-e \\ 0 \end{pmatrix},\
\begin{pmatrix} 0 \\ 2 \end{pmatrix},\
\begin{pmatrix} 4-e \\ 4-e \end{pmatrix},\
\begin{pmatrix} -e \\ e \end{pmatrix}
\]
are all in $\operatorname{Pic}^\chi(X,J)$.
If $e$ is odd, then these generate all of $\operatorname{Pic}(X)$
and therefore $\operatorname{Am}^\chi(X,J)$ is trivial.

If $e$ is even, then we conclude that $\langle 2H, 2F \rangle$
is a subgroup of $\operatorname{Pic}^\chi(X,J)$.
We will show that, in fact, equality holds.
Let $D=aH+bF$ and consider $E:=\chi(D)D=cH+dF$.
If $a$ is odd, then $a=2k+1$ and
\[
\chi(D) \equiv e(k+1) \equiv 0 \mod{2},
\]
so $E$ has even coefficients.
If $a$ is even, then
\[
\chi(D) \equiv -b+1 \mod{2},
\]
which is even when $b$ is odd.
Thus, if $e$ is even, then $c,d$ are even for all choices of $D$.
Thus, $\operatorname{Pic}^\chi(X,J)$ never contains $H$, $H+F$ or $F$.
\end{proof}

\subsection{The del Pezzo surface of degree 6}
\label{sec:dP6}

Here we determine the Amitsur groups of a del Pezzo surface $X$ of degree $6$.
We recall that this is a toric variety obtained by blowing up
$\mathbb{P}^2$ in three non-collinear points $p_1,p_2,p_3$.
Let $E_1,E_2,E_3$ denote the exceptional divisors of the blown up points
and let $L_i$ represent the strict transform of the line passing through
the points $p_j,p_k$ where $i,j,k \in \{1,2,3\}$ are distinct.
The group $\operatorname{TDiv(X)}$ is the free abelian group
with basis $\{E_1,E_2,E_3,L_1,L_2,L_3\}$.
All 6 divisors are in separate linear equivalence classes.

In this case, the fan $\Sigma$ in $\mathbb{R}^2$ is the unique complete
fan with rays
\[
\begin{pmatrix} 1\\0 \end{pmatrix},\
\begin{pmatrix} 0\\1 \end{pmatrix},\
\begin{pmatrix} -1\\-1 \end{pmatrix},\
\begin{pmatrix} -1\\0 \end{pmatrix},\
\begin{pmatrix} 0\\-1 \end{pmatrix},\
\begin{pmatrix} 1\\1 \end{pmatrix},
\]
corresponding to the above basis for $\operatorname{TDiv(X)}$.
The cocharacter lattice $N$ is isomorphic to $\mathbb{Z}^2$
and the matrices
\[
s = \begin{pmatrix} 0 & 1\\ 1 & 0 \end{pmatrix}
\textrm{ and }
r = \begin{pmatrix} 0 & 1\\ -1 & 1 \end{pmatrix}
\]
preserve the fan $\Sigma$.
The fan has automorphism group
$\operatorname{Aut}(\Sigma) = \langle s,r \rangle$
isomorphic to the dihedral group $D_{12}$
of order $12$, which can be thought of as the symmetries of the hexagon
of exceptional lines.
The full automorphism group is
\[
\operatorname{Aut}(X) \cong (\mathbb{C}^\times)^2 \rtimes D_{12},
\]
which coincides with the normalizer of the maximal torus in this case.
In particular, $\PicIm(X) \cong
\operatorname{Aut}(\Sigma)$ and there is no distinction between $J$ and
$W$ here.

Consider the exact sequence
\[
0 \to M \to \operatorname{TDiv}(X) \to \operatorname{Pic}(X) \to 0
\]
where $M$ is the character lattice dual to $N$.
The Picard group has rank $4$ and has basis $\{[H],[E_1],[E_2],[E_3]\}$
where $H$ is the strict transform of a general line on $\mathbb{P}^2$.
We see that $L_i = H - E_j - E_k$ whenever $i,j,k$ are distinct.

By applying Theorem~\ref{thm:toric},
the numerical Amitsur groups of $X$ can be computed.
They are listed in Table~\ref{tab:dP6}. 

\begin{table}[ht]
\centering
\begin{tabular}{|l|l|l|}
\hline
$J$ & Structure & $\operatorname{Am}^\chi(X,J)$\\
\hline
$1$ & $1$ & 0 \\
$\langle s \rangle$ & $C_2$ & 0 \\
$\langle sr^3 \rangle$ & $C_2$ & 0 \\
$\langle r^3 \rangle$ & $C_2$ & $(\mathbb{Z}/2\mathbb{Z})^2$ \\
$\langle r^2 \rangle$ & $C_3$ & $\mathbb{Z}/3\mathbb{Z}$ \\
$\langle s, r^3 \rangle$ & $C_2^2$ & $\mathbb{Z}/2\mathbb{Z}$ \\
$\langle s, r^2 \rangle$ & $S_3$ & 0 \\
$\langle sr^3, r^2 \rangle$ & $S_3$ & $\mathbb{Z}/3\mathbb{Z}$ \\
$\langle r \rangle$ & $C_6$ & 0 \\
$\langle s,r \rangle$ & $D_{12}$ & 0 \\
\hline
\end{tabular}
\caption{Numerical Amitsur groups of $dP_6$}
\label{tab:dP6}
\end{table}

\begin{remark} \label{rem:MackeyFunctor}
Observe that both $\operatorname{Am}^\chi(X,D_{12})$
and $\operatorname{Am}^\chi(X,1)$ are trivial, but many of the
intermediate subgroups have nontrivial Amitsur groups.
This demonstrates that the assumption
that $\PicIm(X,G)=\PicIm(X,H)$ cannot be removed
in Proposition~\ref{prop:PicIm}.
\end{remark}

\begin{remark}
In the cases where
$\operatorname{Am}^\chi(X,J) \cong \mathbb{Z}/3\mathbb{Z}$,
the exceptional divisors $E_1,E_2,E_3$ can be equivariantly blown down
to produce $\mathbb{P}^2$.
In the other cases where
$\operatorname{Am}^\chi(X,J)$ is non-trivial,
a pair of exceptional divisor $E_i,L_i$ can be equivariantly blown down
to produce $\mathbb{P}^1 \times \mathbb{P}^1$.
This indicates how one can find groups $G$ that realize non-trivial
values for the ordinary Amitsur group $\operatorname{Am}(X,G)$ in these
cases.
Applying the contrapositive,
we recover the fact that the Amitsur group is always trivial for a
\emph{$G$-minimal} del Pezzo $G$-surface
of degree 6
(see \cite[Proposition~A.7]{Blanc.Cheltsov.23}).
\end{remark}

\begin{remark}
It is tempting to try to compute the numerical Amitsur
group via the quotient of the group
\[
P:= \langle \chi(\mathcal{L})\mathcal{L} \mid \mathcal{L} \in
\operatorname{Pic}(X)^G \rangle \subseteq \operatorname{Pic}(X)^G
\]
instead of the definition of $\operatorname{Pic}^\chi(X,G)$
involving $G$-orbits from \eqref{eq:numPic_def}.
However, they do not give the same result.

Every line bundle in $P$ is contained in $\operatorname{Pic}^\chi(X,G)$,
but not conversely.  Indeed, consider a group $G$ such that
\[
J=W = \langle sr^3 \rangle = \left\langle
\begin{pmatrix} 0 & -1\\ -1 & 0 \end{pmatrix} \right\rangle \cong C_2.
\]
We see that
\[
\operatorname{TDiv}(X)^W = \langle
E_1+L_2,\ E_2+L_1,\ E_3+L_3
\rangle \cong \mathbb{Z}^3,
\]
which maps onto
\[
\operatorname{Pic}(X)^J=\langle A:=H-E_3,\ B:=H-E_1-E_2-E_3\rangle
\cong \mathbb{Z}^2.
\]
Using Riemann-Roch and intersection theory, or using toric methods,
we determine that
\[
\chi(mA+nB)=2mn-n^2+m+n+1.
\]
Thus, for $L=mA+nB$, we see that $\chi(L)[L]$ is given by
\[
\left(2m^2n-n^2m+m^2+mn+m\right)A
+\left(2mn^2-n^3+mn+n^2+n\right)B .
\]
Observe that the coefficient of $A$ can be rewritten as
\[
2m^2n+m(n-n^2)+(m^2+m),
\]
which is even for all values of $m,n$.
Therefore $A$ is not in the span of 
$\chi(L)L$ for $L \in \operatorname{Pic}(X)^G$.
However, $A = E_1 + sr^3(E_1)$ and $\chi(E_1)=-1$.
Thus $A \in \operatorname{Pic}^\chi(X,G)$.
\end{remark}


\bibliographystyle{alpha}

\begin{thebibliography}{BCDP23}

\bibitem[BCDP23]{Blanc.Cheltsov.23}
J{\'e}r{\'e}my Blanc, Ivan Cheltsov, Alexander Duncan, and Yuri Prokhorov.
\newblock Finite quasisimple groups acting on rationally connected threefolds.
\newblock {\em Mathematical Proceedings of the Cambridge Philosophical Society}, 174(3):531--568, May 2023.

\bibitem[Bri18]{Brion18}
Michel Brion.
\newblock Linearization of algebraic group actions.
\newblock In {\em Handbook of Group Actions. {{Vol}}. {{IV}}}, volume~41 of {\em Adv. {{Lect}}. {{Math}}. ({{ALM}})}, pages 291--340. Int. Press, Somerville, MA, 2018.

\bibitem[Bro82]{Bro82Cohomology}
Kenneth~S. Brown.
\newblock {\em Cohomology of Groups}, volume~87 of {\em Graduate Texts in Mathematics}.
\newblock Springer-Verlag, New York, 1982.

\bibitem[CLS11]{CoxLitSch11Toric}
David~A. Cox, John~B. Little, and Henry~K. Schenck.
\newblock {\em Toric Varieties}, volume 124 of {\em Graduate Studies in Mathematics}.
\newblock American Mathematical Society, Providence, RI, 2011.

\bibitem[Cox95]{Cox95The-homogeneous}
David~A. Cox.
\newblock The homogeneous coordinate ring of a toric variety.
\newblock {\em Journal of Algebraic Geometry}, 4(1):17--50, 1995.

\bibitem[CS21]{Colliot-Thelene.Skorobogatov21}
Jean-Louis {Colliot-Th{\'e}l{\`e}ne} and Alexei~N. Skorobogatov.
\newblock {\em The {{Brauer-Grothendieck}} Group}, volume~71 of {\em Ergebnisse Der Mathematik Und Ihrer Grenzgebiete. 3. {{Folge}}. {{A}} Series of Modern Surveys in Mathematics [Results in Mathematics and Related Areas. 3rd Series. {{A}} Series of Modern Surveys in Mathematics]}.
\newblock Springer, Cham, [2021] {\copyright}2021.

\bibitem[Dav00]{Davenport}
Harold Davenport.
\newblock {\em Multiplicative Number Theory}, volume~74 of {\em Graduate Texts in Mathematics}.
\newblock Springer-Verlag, New York, 3 edition, 2000.

\bibitem[Dol99]{Dol99Invariant}
Igor~V. Dolgachev.
\newblock Invariant stable bundles over modular curves {{X}}(p).
\newblock In {\em Recent Progress in Algebra ({{Taejon}}/{{Seoul}}, 1997)}, volume 224 of {\em Contemp. {{Math}}.}, pages 65--99. Amer. Math. Soc., Providence, RI, 1999.

\bibitem[Dun16]{Duncan16}
Alexander Duncan.
\newblock Twisted forms of toric varieties.
\newblock {\em Transformation Groups}, 21(3):763--802, September 2016.

\bibitem[Kle66]{Kleiman66}
Steven~L. Kleiman.
\newblock Toward a {{Numerical Theory}} of {{Ampleness}}.
\newblock {\em The Annals of Mathematics}, 84(3):293, November 1966.

\bibitem[KT22]{Kresch.Tschinkel22}
Andrew Kresch and Yuri Tschinkel.
\newblock Cohomology of finite subgroups of the plane {{Cremona}} group, March 2022.

\bibitem[Lie17]{Liedtke17}
Christian Liedtke.
\newblock Morphisms to {{Brauer}}--{{Severi Varieties}}, with {{Applications}} to {{Del Pezzo Surfaces}}.
\newblock In Fedor Bogomolov, Brendan Hassett, and Yuri Tschinkel, editors, {\em Geometry {{Over Nonclosed Fields}}}, pages 157--196. Springer International Publishing, Cham, 2017.

\bibitem[MFK94]{MumFogKir94Geometric}
D.~Mumford, J.~Fogarty, and F.~Kirwan.
\newblock {\em Geometric Invariant Theory}, volume~34 of {\em Ergebnisse Der Mathematik Und Ihrer Grenzgebiete (2) [{{Results}} in Mathematics and Related Areas (2)]}.
\newblock Springer-Verlag, Berlin, 3 edition, 1994.

\bibitem[Mum08]{Mumford08}
David Mumford.
\newblock {\em Abelian Varieties}, volume~5 of {\em Tata Institute of Fundamental Research Studies in Mathematics}.
\newblock Tata Institute of Fundamental Research, Bombay; by Hindustan Book Agency, New Delhi, 2008.

\bibitem[Nil06]{Nill06}
Benjamin Nill.
\newblock Complete toric varieties with reductive automorphism group.
\newblock {\em Mathematische Zeitschrift}, 252(4):767--786, April 2006.

\bibitem[{OEI}25]{oeis}
{OEIS Foundation Inc.}
\newblock The {{On-Line Encyclopedia}} of {{Integer Sequences}}, 2025.

\bibitem[PZ24]{Pirutka.Zhang24}
Alena Pirutka and Zhijia Zhang.
\newblock Computing the equivariant {{Brauer}} group, October 2024.

\end{thebibliography}

\end{document}